\documentclass[a4paper,twoside,10pt]{article}

\usepackage[a4paper,left=3cm,right=3cm, top=3cm, bottom=3cm]{geometry}
\usepackage{enumitem}
\setlist[itemize]{noitemsep, nolistsep}
\usepackage{amsfonts}
\usepackage{comment}
\usepackage{graphicx}
\usepackage{epstopdf}
\usepackage{scalerel}
\usepackage{framed,multirow}
\usepackage{etoolbox}
\usepackage{amscd,amsmath,amssymb,mathrsfs,bbm,listings}
\usepackage{amsthm}
\usepackage{tikz} 
\usetikzlibrary{arrows,automata,calc} 
\usepackage{epsfig}
\usepackage{multirow}
\usepackage{subfig}

\newtheorem{theorem}{Theorem}[section]

\newtheorem{proposition}[theorem]{Proposition}

\newtheorem{remark}[theorem]{Remark}
\newtheorem{definition}[theorem]{Definition}

\usepackage[hidelinks]{hyperref}
\hypersetup{
    colorlinks=true,
    pdfborder={0 0 0},
    allcolors =blue,
}

\numberwithin{equation}{section}

\ifpdf
  \DeclareGraphicsExtensions{.eps,.pdf,.png,.jpg}
\else
  \DeclareGraphicsExtensions{.eps}
\fi

\def\dx{\,\mathrm{d}\bx} 
\def\dS{\,\mathrm{d}S}
\def\dV{\,\mathrm{d}V}

\def\Deltax{\Delta_{\bx}}
\def\nablax{\nabla_{\bx}}
\def\QT{Q_T}
\newcommand\QTrefftz[2]{\mathbb{Q}\mathbb{T}_{#1}\left(#2\right)}
\DeclareMathOperator{\spn}{span}

\def\calS{\mathcal{S}}

\newcommand{\dfdt}[1]{\partial_t{#1}}

\newcommand{\EFC}[2]{\calC^{#1}\left({#2}\right)}
\newcommand{\ESOBOLEV}[2]{H^{#1}\left({#2}\right)}

\newcommand{\DerA}[2]{D^{#1}{#2}}

\newcommand{\Taylor}[3]{T_{#1}^{#2}\left[#3\right]}
\newcommand{\ETaylor}[3]{\widetilde{T}_{#1}^{#2}\left[#3\right]}
\newcommand{\AvTaylor}[2]{\calQ^{#1}\left[#2\right]}
\newcommand{\EAvTaylor}[2]{\widetilde{\calQ}^{#1}\left[#2\right]}
\newcommand{\Conjugate}[1]{\overline{#1}}
\newcommand{\jump}[1]{\left[\!\left[#1\right]\!\right]}

\newcommand{\ORDER}[1]{{\mathcal O}\left(#1\right)}

\numberwithin{equation}{section}

\newlength{\dhatheight}

\allowdisplaybreaks[4]
\newcommand{\normalDer}{\partial_{\mathbf{n_x}}}
\newcommand{\nF}{\vec{\bn}_F}
\newcommand{\nFt}{n_F^t}
\newcommand{\nFx}{\vec{\bn}_F^{\bx}}

\newcommand{\Uu}[1]{{\mathbf{#1}}}

\newcommand{\bj}{{\boldsymbol j}}
\newcommand{\bjx}{{\boldsymbol j_\bx}}

\newcommand{\bjxl}{j_{x_\ell}}

\newcommand{\bT}{{\Uu T}}
\newcommand{\bTp}{{\mathbb{T}}}

\newcommand{\calC}{{\mathcal C}}

\newcommand{\calT}{{\mathcal T}}

\newcommand{\calF}{{\mathcal F}}

\newcommand{\calQ}{{\mathcal Q}}
\newcommand{\Fh}{\calF_h}
\newcommand{\Kx}{K_{\bx}}
\newcommand{\Kt}{K_t}
\newcommand{\hK}{h_{K}}
\newcommand{\hKx}{h_{K_\bx}}
\newcommand{\hKt}{h_{K_t}}
\newcommand{\Th}{{\calT_h}}

\newcommand{\deK}{{\partial K}}

\newcommand{\uhp}{{\psi_{hp}}}
\newcommand{\uhpT}{{\psi_{hp}^{-}}}
\newcommand{\shp}{{s_{hp}}}
\newcommand{\cs}{{\conj{s_{hp}}}}

\newcommand{\IN}{\mathbb{N}}
\newcommand{\IR}{\mathbb{R}}
\newcommand{\po}{\partial \Omega}
\newcommand{\gD}{{g_{\mathrm D}}}

\newcommand{\oon}{\;\text{on}\;}
\newcommand{\bx}{{\Uu x}}

\newcommand{\bn}{{\Uu n}}

\newcommand{\bN}{{\Uu N}}

\newcommand{\bz}{{\Uu z}}
		
\newcommand{\IC}{\mathbb{C}}
\newcommand{\IP}{\mathbb{P}}
\newcommand{\IT}{\mathbb{T}}

 \newcommand{\mvl}[1]{\left\{ \!\left\{#1\right\}\!\right\}}  
 \newcommand{\FT}{{\Fh^T}}
\newcommand{\FO}{{\Fh^0}}

\newcommand{\FD}{{\Fh^{\mathrm D}}}

\newcommand{\rtime}{{\mathrm{time}}}
\newcommand{\rspace}{{\mathrm{space}}}
\newcommand{\Fspa}{{\Fh^\rspace}}
\newcommand{\Ftime}{{\Fh^\rtime}}
\newcommand{\bThp}{{\IT_{hp}}} 
\newcommand*{\conj}[1]{\overline{#1}}

\DeclareMathOperator{\diam}{diam} 
\DeclareMathOperator{\esssup}{ess\,sup}
\DeclareMathOperator{\essinf}{ess\,inf}
\newcommand{\cbA}[2]{{{\mathcal{A}}}\left({#1};\ {#2}\right)}
\newcommand{\be}{\boldsymbol{e}}

\newcommand*{\abs}[1]{\left|#1\right|}
\newcommand{\Tnorm}[2]{|||#1|||_{#2}}

\newcommand*{\Norm}[2]{\left\|#1\right\|_{#2}}
\newcommand{\DG}{_{\mathrm{DG}}}
\newcommand{\DGp}{_{\mathrm{DG^+}}}

\newcommand{\deO}{{\partial\Omega}}

\newcommand{\tr}{\mathrm{tr}}

\title{On polynomial Trefftz spaces for the linear time-dependent Schr\"odinger equation\thanks{The authors have been funded by the Austrian Science Fund (FWF) through the projects F~65
and P~33477 (I. Perugia),
by the 
Italian 
Ministry of University and Research 
through the PRIN project ``NA-FROM-PDEs'', from PNRR-M4C2-I1.4-NC-HPC-Spoke6.
(A. Moiola, S. G\'omez),
and by DFG SFB 1456 project 432680300 (P. Stocker). S. G\'omez acknowledges the kind hospitality of the Erwin Schr\"odinger International Institute for Mathematics and Physics (ESI), where part of this research was developed.}}

\author{\large{Sergio G\'omez\footnotemark[3]\;\thanks{Erwin Schr\"odinger Institute for Mathematics and Physics, University of Vienna.},\; Andrea Moiola\thanks{Department of Mathematics, University of Pavia, Via Ferrata 5, 27100 Pavia, Italy (sergio.gomez01@universitadipavia.it, andrea.moiola@unipv.it)},\; Ilaria Perugia\thanks{Faculty of Mathematics, University of Vienna,  Oskar-Morgenstern-Platz 1, 1090, Vienna, Austria (ilaria.perugia@univie.ac.at)},\; Paul Stocker\thanks{Institut f\"ur Numerische und Angewandte Mathematik, University of G\"ottingen, Lotzestr. 16-18, 37083 G\"ottingen, Germany (p.stocker@math.uni-goettingen.de)}}}

\begin{document}

\maketitle

\begin{abstract}
\noindent We study the approximation properties of complex-valued polynomial Trefftz spaces for the~$(d+1)$-dimensional linear time-dependent Schr\"odinger equation. More precisely, we prove that for the space--time Trefftz discontinuous Galerkin variational formulation proposed by G\'omez, Moiola~(SIAM. J. Num. Anal. 60(2): 688--714, 2022), the same~$h$-convergence rates as for polynomials of degree~$p$ in~$(d + 1)$ variables can be obtained in a mesh-dependent norm by using a space of Trefftz polynomials of anisotropic degree.
For such a space, the dimension is equal to that of the space of polynomials of degree~$2p$ in~$d$ variables, and bases are easily constructed. 
\end{abstract}


\section{Introduction\label{SEC::INTRODUCTION}}
 \noindent Trefftz methods are characterized by 
 constructing approximation spaces that lie in the kernel of the target differential operator.
 This leads to a substantial reduction in the total number of degrees of freedom without loss of optimal approximation properties, and can be used to remove the volume terms in the variational formulation. 
 Trefftz discontinuous Galerkin (Trefftz-DG) variational formulations have been designed for several models, see e.g.~\cite{Qin05,Hiptmair_Moiola_Perugia_2016,Hiptmair_Moiola_Perugia_2013,Egger_Kretzchmar_Schnepp_Weiland_2015,Banjai_Georgoulis_Lijoka_2017,Moiola_Perugia_2018,Gomez_Moiola_2022}. Typically, well-posedness and quasi-optimality are proven for such formulations on very general discrete Trefftz spaces. Nonetheless, \emph{a priori} error estimates depend on the specific choice of the discrete Trefftz space, which is desired to possess good approximation properties.
 In fact, one of the main advantages of Trefftz methods is that, 
 for BVPs with zero volume source term,
 they allow for spaces with the same asymptotic accuracy as standard full polynomial spaces, but with much smaller dimension. Such a property translates into a substantial reduction in the total number of degrees of freedom. When all the derivatives in the differential operator are of the same order (e.g., the Laplace's equation, the wave equation, the time-dependent Maxwell's equations), the space of Trefftz polynomials of a certain maximum degree delivers the same convergence rates as the full polynomial space of the same degree~\cite[Lemma 1]{Moiola_Perugia_2018}. This is due to the  fact that the averaged Taylor polynomial of the exact solution belongs to the polynomial Trefftz space, thus ensuring good approximation properties. On the contrary, when the differential operator includes derivatives of different orders (e.g., the Helmholtz's equation, the time-harmonic Maxwell's equation, the Schr\"odinger equation), non-polynomial spaces have been used in the literature, as no Trefftz subspace of polynomials of a given degree delivers the same accuracy as the full polynomial space of the same degree. For their simple shape and closed-form integration formulas, the preferred type of non-polynomial Trefftz spaces for time-harmonic problems are plane waves, see~\cite{Hiptmair_Moiola_Perugia_2013,Gomez_Moiola_2022}. Unfortunately, the use of plane waves Trefftz spaces brings some consequences: ill-conditioned matrices in the resulting linear system; lack of general
tools for their analysis, and therefore need of developing novel tools for each problem; dependence of the design of such spaces on the physical dimension of the problem; dimension-dependent conditions on the parameters defining their basis functions in order to preserve accuracy. The use of polynomial Trefftz spaces would mitigate or overcome these issues.

In this work, we show that optimal \emph{a priori} error estimates with \emph{polynomial} Trefftz spaces can be obtained for a Trefftz-DG discretization of a PDE that includes derivatives of different orders, namely the time-dependent Schr\"odinger equation in~$d$ space dimensions. For the variational formulation in~\cite{Gomez_Moiola_2022},
we prove that the same asymptotic accuracy as for full polynomials of maximum degree~$p$ in~$\IR^{d+1}$, is obtained if Trefftz polynomials of degree~$2p$ in~$\IR^{d+1}$ are used.
We also propose a practical way to construct an explicit basis for the polynomial Trefftz space in arbitrary dimensions and prove that its dimension is equal to that of the space of polynomials of degree~$2p$ in~$\IR^d$.

We focus on the dimensionless linear Schr\"odinger equation on the space--time cylinder~$\QT := \Omega \times I$, where~$\Omega \subset \IR^d$ ($d\in\IN$) is an open,  
bounded polytopic domain with Lipschitz boundary~$\po$, and~$I =
(0, T)$ for 
some final time $T > 0$
\begin{equation}
\label{EQN::SCHRODINGER-EQUATION}
\begin{split}
\calS \psi := i \dfdt{\psi} + \frac{1}{2}\Deltax \psi = 0  & \quad \mbox{ in }\ \QT,\\
\psi = \gD   & \quad \oon\ \deO\times I,\\
\psi(\bx, 0) = \psi_0(\bx)  & \quad \oon\ \Omega.
\end{split}
\end{equation}
Here~$i$ is the imaginary unit, and the Dirichlet~($\gD$) and initial condition~($\psi_0$) data are given functions.

\section{Ultra-weak Trefftz discontinuous Galerkin formulation\label{SECT::ULTRA-WEAK-DG}}
\noindent Let~$\Th$ be a non-overlapping prismatic partition of~$\QT$, i.e., each element $K \in \Th$ can be written as~$K = \Kx \times \Kt$ for a~$d$-dimensional polytope~$\Kx \subset \Omega$ and a time interval~$\Kt \subset I$.
We adopt the notation in~\cite[Sect. 2.1]{Gomez_Moiola_2023},  $\hKx = \diam(\Kx)$, $\hKt = \abs{\Kt}$ and $\hK = \diam(K)=(\hKx^2+\hKt^2)^{1/2}$.
Any intersection $F=\deK_1\cap \deK_2$ or $F=\deK_1\cap\partial Q_T$, for
$K_1,K_2\in\Th$, that has positive $d$-dimensional measure and is contained in a $d$-dimensional hyperplane is called ``mesh facet".
We denote by~$\nF = (\nFx, \nFt) \in \IR^{d+1}$ one of the two unit normal vectors orthogonal to~$F$, with either~$\nFt = 0$ or $\nFt = 1$.
We assume that each internal mesh facet $F$ is either
a space-like facet (if~$\nFt = 1$), or 
a time-like facet (if~$\nFt = 0$).
We further denote by~$ \Fspa$ and~$\Ftime$ the union of all the internal space-like and time-like facets, respectively, and by 
\begin{equation*}
\Fh := \bigcup_{K \in \Th} \deK,\quad
\FO := \Omega \times \left\{0\right\}, \quad 
\FT := \Omega \times \left\{T\right\}, \quad 
\FD := \deO  \times (0, T).
\end{equation*}

\noindent Given a finite-dimensional subspace~$\bThp\left(\Th\right)$ of the Trefftz space
$$\bT(\Th) := \prod_{K \in \Th} \bT(K), \quad \bT(K) := \left\{w \in \ESOBOLEV{1}{\Kt; L^2(\Kx)} \cap L^2\left(\Kt;
\ESOBOLEV{2}{\Kx}\right) : i\dfdt w + \frac12 \Deltax w  = 0\right\},
$$
by employing the standard DG notation for the averages~$\mvl{\cdot}$ and 
space~$\jump{\cdot}_{\bN}$ and time~$\jump{\cdot}_t$ jumps for 
piecewise smooth complex-valued scalar and vector fields, the ultra-weak Trefftz-DG variational formulation for the Schr\"odinger equation introduced in~\cite{Gomez_Moiola_2022} (see also~\cite{Gomez_Moiola_2023} for a quasi-Trefftz version)  reads 
\begin{equation}
\label{EQN::VARIATIONAL-DG}
\mbox{Seek }\uhp \in \bThp(\Th) \mbox{ such that} \quad  \cbA{\uhp}{\shp} = 
\ell(\shp) \quad \forall \shp \in \bThp(\Th),
\end{equation}
where
\begin{align*}
\cbA{\uhp}{\shp}  := & 
i \left(\int_{\Fspa} \uhpT \jump{\cs}_t  \dx + \int_{\FT}  
\uhp \cs \dx \right) \\
&  + \frac{1}{2} \int_{\Ftime} \Bigg(\mvl{\nablax \uhp} \cdot \jump{\cs}_{\bN} + i \alpha 
\jump{\uhp}_{\bN} \cdot \jump{\cs}_{\bN} 
 \\
&  - \mvl{\uhp} \jump{\nablax
\cs}_{\bN}  + i \beta 
\jump{\nablax \uhp}_{\bN} \jump{\nablax \cs}_{\bN}\Bigg) \dS  \\
& + \frac{1}{2}  \int_{\FD} 
\left(\normalDer \uhp + i \alpha
\uhp \right)\cs \dS,  \\
\ell(\shp)  := & i \int_{\FO}  \psi_0 \cs \dx + \frac{1}{2} \int_{\FD} \gD \left(\normalDer  
\cs + i \alpha\cs\right) \dS,
\end{align*}
where~$\Conjugate{\, \cdot\, }$ denotes the complex conjugate, and the mesh-dependent stabilization functions~$\alpha$ and~$\beta$ are set as

\begin{equation*}
\begin{gathered}
\alpha|_{F} = h_{F_\bx}^{-1}
\quad \forall F \subset \Ftime\cup\FD, 
\qquad \beta|_{F} = h_{F_\bx} \quad \forall F \subset \Ftime,
\end{gathered}
\end{equation*} 
with
\begin{equation*}
\begin{cases}
h_{F_\bx}=h_{\Kx} &  \text{if } F\subset \deK\cap \FD,\\
\min\{h_{\Kx^1},h_{\Kx^2}\}\le h_{F_\bx}\le
\max\{h_{\Kx^1},h_{\Kx^2}\} & \text{if } F=K^1 \cap K^2 \subset \Ftime.
\end{cases}
\end{equation*}
The sesquilinear form~$\cbA{\cdot}{\!\cdot}$ in the variational formulation~\eqref{EQN::VARIATIONAL-DG} induces the following norms on~$\bT(\Th)$
\begin{align*}
\Tnorm{w}{\DG}^2  : =  &\, \frac12 \Big(\Norm{\jump{w}_t}{L^2(\Fspa)}^2 + 
\Norm{w}{L^2(\FT \cup \FO)}^2  + \Norm{\alpha^\frac12\jump{w}_{\bN}}{L^2(\Ftime)^d}^2 \\
& \qquad + \Norm{\beta^\frac12 \jump{\nablax w}_{\bN}}{L^2(\Ftime)}^2 + \Norm{\alpha^\frac12w}{L^2(\FD)}^2 \Big),
\\
\Tnorm{w}{\DGp}^2  : = &\, \Tnorm{w}{\DG}^2 + \frac12\Big( \Norm{w^-}{L^2(\Fspa)}^2 + \Norm{\alpha^{-\frac12}\mvl{\nablax w}}{L^2(\Ftime)^d}^2 \\
& \qquad + \Norm{\alpha^{-\frac12} \normalDer  
w }{L^2(\FD)} + \Norm{\beta^{-\frac12}\mvl{w}}{L^2(\Ftime)}^2\Big).
\nonumber
\end{align*}

In what follows we use the standard multi-index notation for partial derivatives and monomials, adapted to the space--time setting, see~\cite[Sect.~4.1]{Gomez_Moiola_2022} for more details. Henceforth, we also assume that:
\begin{itemize}
\item {\bf Uniform star-shapedness}: 
There exists~$0 < \rho \leq \frac{1}{2}$ such that, each element~$K \in \Th$ is star-shaped with respect to the ball~$B:= B_{\rho \hK}(\bz_K, s_K)$ centered at~$(\bz_K, s_K) \in K$ and
with radius~$\rho \hK$.
\item {\bf Local quasi-uniformity in space}: there exists a number~$\mathsf{lqu}(\Th)>0$ such that~$h_{\Kx^1}\le h_{\Kx^2}\, \mathsf{lqu}(\Th)$ for all~$K^1 = \Kx^1 \times \Kt^1 ,K^2 = \Kx^2 \times \Kt^2\in\Th$ such that~$K^1\cap K^2$ 
has positive~$d$-dimensional measure.
\end{itemize}

We recall the following two results from~\cite{Gomez_Moiola_2022}, which are valid for any choice of the discrete Trefftz space~$\bThp(\Th)$.

\begin{theorem}[{Quasi-optimality \cite[Thm.~3.4]{Gomez_Moiola_2022}}] \label{THEOREM::WELL-POSEDNESS} 
For any finite-dimensional subspace~$\bThp(\Th)$ of~$\bT(\Th)$, there exists 
a unique function~$\uhp \in \bThp(\Th)$ satisfying the variational formulation
\eqref{EQN::VARIATIONAL-DG}. Additionally, the following 
quasi-optimality bound is satisfied:
\begin{equation*}
\label{EQN::QUASI-OPTIMALITY}
\Tnorm{\psi - \uhp}{\DG} \leq 3 \inf_{\shp \in \bThp(\Th)}\Tnorm{\psi - 
\shp}{\DGp}.
\end{equation*}
\end{theorem} 

\begin{proposition}[{\cite[Prop.~4.11]{Gomez_Moiola_2022}}]\label{PROP::DGp-BOUND}
For all~$\varphi\in\bT(\Th)$,
we have
\begin{align*}
\Tnorm{\varphi}{\DGp}^2 \leq  3C_{\tr} 
 \!\!\!\!\!\! \sum_{K = \Kx \times \Kt \in \Th} \Bigg[&
\left(\hKt^{-1} + \mathrm{a}_K^2 \hKx^{-1}\right) \Norm{\varphi}{L^2(K)}^2 + \hKt
\Norm{\partial_t \varphi}{L^2(K)}^2 \\
& + \left(\mathrm{a}^2_K \hKx + \mathrm{b}_K^2 \hKx^{-1}\right)
\Norm{\nablax \varphi}{L^2(K)^d}^2  + 
\mathrm{b}_K^2 \hKx \Norm{D_{\bx}^2 \varphi}{L^2(K)^{d\times d}}^2\Bigg],
\end{align*}
where
\[
\begin{tabular}{ll}  
& $\mathrm{a}_K := \max\Bigg\{\Bigg(\underset{\partial K \cap (\Ftime \cup 
\FD)}{\esssup}\alpha \Bigg)^{1/2},\ 
\Bigg(\underset{\partial K \cap \Ftime}{\essinf} \, \beta \Bigg)^{-1/2}\Bigg\},$ \\
& $\mathrm{b}_K := \max\Bigg\{\Bigg(\underset{\partial K \cap (\Ftime \cup 
\FD)}{\essinf} \alpha \Bigg)^{-1/2},\ 
\Bigg(\underset{\partial K \cap \Ftime}{\esssup}\, \beta \Bigg)^{1/2}\Bigg\}$.
\end{tabular}
\]
\end{proposition}
From the above results, 
it is possible to derive error estimates for the Trefftz-DG method~\eqref{EQN::VARIATIONAL-DG} by simply studying the approximation properties of the local discrete spaces~$\bThp(K)$ for functions in~$\bT(\Th)$. In particular, our analysis for the polynomial Trefftz space relies on the approximation properties of some Taylor-type polynomials. 

On an open and bounded set~$\Upsilon \subset\IR^{d+1}$, we denote by~$\Taylor{(\bz,s)}{m}{\varphi}$
the Taylor polynomial of order~$m\in\IN$ (and degree~$m - 1$), centered at some~$(\bz,s)\in \Upsilon$,
of a given  function~$\varphi \in \EFC{m - 1}{\Upsilon}$. We also denote by~$\AvTaylor{m}{\varphi}$ the averaged Taylor polynomial of order~$m$ of~$\varphi \in \ESOBOLEV{m-1}{\Upsilon}$, see~\cite[Sect.~4.1]{Brenner_Scott_2007}.

Given~$p \in \IN$, let~$\IP^p_{\bT}(K) := \bT(K) \cap \IP^p(K) $ be the space of Trefftz polynomials of maximum degree~$p$. From the above results,  
it would be desirable to prove that, for all~$\psi \in \bT(K)$, either~$\Taylor{(\bz, s)}{p + 1}{\psi} \in \IP^p_{\bT}(K)$ for some~$(\bz, s) \in K$ or~$\AvTaylor{p + 1}{\psi} \in \IP^p_{\bT}(K)$, as this would be enough to achieve optimal~$h$-convergence rates for sufficiently smooth solutions. Unfortunately, this is not true; e.g., for~$d = 1$ and~$p = 1$ the function~$\psi(x, t) = \exp\left(x + \frac{it}{2}\right)$ satisfies the Schr\"odinger equation~\eqref{EQN::SCHRODINGER-EQUATION}; however, $\Taylor{(0, 0)}{2}{\psi} = 1 + x + \frac{it}{2}$ does not belong to~$\IP^1_{\bT}(K)$. Instead, in Proposition~\ref{PROP::EXTENDED-TAYLOR-POLYNOMIALS} below, we show that the following extended Taylor polynomials belong to~$\IP^{2p}_{\bT}(K)$.

\begin{definition}[Extended Taylor polynomial]
Given~$p\in \IN$, and an open and bounded set~$\Upsilon = \Upsilon_\bx \times \Upsilon_t \subset \IR^{d+1}$, for each complex-valued~$\varphi \in \EFC{p}{\Upsilon_t; \EFC{2p}{\Upsilon_\bx}}$ we define its extended Taylor polynomial of degree~$2p$ centered at some~$(\bz, t) \in \Upsilon$ as
\begin{equation*}
\label{DEF::EXTENDED-TAYLOR}
    \ETaylor{(\bz, t)}{2p}{\varphi} = \Taylor{(\bz, t)}{p+1}{\varphi} + \!\!\!\!\!\! \sum_{
    \scriptsize
    \begin{split}
        2j_t + |\bjx| \leq 2p \\[-0.5em]
        p + 1 \leq j_t + |\bjx|
    \end{split}
    }
    \frac{1}{\bjx! j_t!}\DerA{\bj}{\varphi(\bz, s)} (\bx - \bz)^\bjx (t - s)^{j_t} \qquad \bj = (\bjx, j_t) \in \IN^{d + 1}.
\end{equation*}
\end{definition}

\begin{definition}[Extended averaged Taylor polynomial]
\label{DEF::EXTENDED-AVERAGED-TAYLOR-POLYNOMIAL}
Let~$\Upsilon = \Upsilon_\bx \times \Upsilon_t\subset \IR^{d + 1}$ be an open and bounded set, with diameter~$h_\Upsilon$, star-shaped with 
respect to the ball $B:= B_{\rho h_\Upsilon}(\bz, s)$ centered at~$(\bz, s) \in \Upsilon$ and 
with radius~$\rho h_\Upsilon$, for some $0 < \rho \leq \frac{1}{2}$. For each complex-valued~$\varphi \in 
\ESOBOLEV{p}{\Upsilon_t; \ESOBOLEV{2p}{\Upsilon_{\bx}}}$, we define its extended averaged Taylor polynomial of degree~$2p$ as
\begin{equation*}
\label{EQN::AVERAGED-TAYLOR-POLYNOMIAL}
\begin{split}
\EAvTaylor{2p}{\varphi}(\bx, t) & := 
\frac{1}{\abs{B}} \int_{B} \ETaylor{(\bz,s)}{2p}{\varphi}(\bx, t) \dV(\bz,s) \\
& = \AvTaylor{p+1}{\varphi} + \!\!\!\!\! \sum_{
    \scriptsize
    \begin{split}
        2j_t + |\bjx| \leq 2p \\[-0.5em]
        p + 1 \leq j_t + |\bjx|
    \end{split}
    }
     \frac{1}{\bjx! j_t! \abs{B}} \int_B \DerA{\bj}{\varphi(\bz, s)} (\bx - \bz)^\bjx (t - s)^{j_t} \dV(\bz, s).
\end{split}     
\end{equation*}
\end{definition}

\begin{remark}[Heat polynomials]
In the real-valued case, these extended Taylor polynomials can be used to analyze the approximation properties of the so-called heat polynomials, i.e., the polynomial solutions to the heat equation~\cite{Rosenbloom_Widder_1959}. In fact, Proposition~\ref{PROP::EXTENDED-TAYLOR-POLYNOMIALS} below can be easily extended to the heat equation. However, as the heat operator~$\mathcal{L}(\cdot) := (\dfdt - \Deltax)(\cdot)$ is not self-adjoint, the setting of~Theorem~\ref{THEOREM::WELL-POSEDNESS} cannot be used to analyze the ultra-weak DG formulation of the heat equation~\cite[Remark 2.2]{Gomez_Moiola_2023}. This is due to the fact that the corresponding bilinear form is not coercive.
\end{remark}

\begin{proposition}
\label{PROP::EXTENDED-TAYLOR-POLYNOMIALS}
    Let~$\Upsilon$ be an open and bounded set in~$\IR^{d+1}$ satisfying the assumptions in Definition~\ref{DEF::EXTENDED-AVERAGED-TAYLOR-POLYNOMIAL}. For all~$\psi \in \bT(\Upsilon)$ sufficiently regular, both~$\ETaylor{(\bz, s)}{2p}{\psi}$ and~$\EAvTaylor{2p}{\psi}$ belong to~$\IP^{2p}_{\bT}(\Upsilon)$.
\end{proposition}
\begin{proof}
Denote by~$\{\be_\ell\}_{\ell = 1}^d $ the canonical basis of~$\IR^d$. Since~$\psi \in \bT(\Upsilon)$, we have that 
\begin{align}\label{eq:tsch}
\sum_{\ell = 1}^d \DerA{(\bjx + 2\be_\ell, j_t)}{\psi}(\bz, s) = -2i \DerA{(\bjx, j_t + 1)}\psi(\bz, s)  \qquad \forall \bj = (\bjx, j_t) \in \IN^{d+1}, \ \abs{\bj} \geq 0.
\end{align}
Applying this to the Taylor expansion of $\psi$ and using straightforward multi-index manipulations, we get
\begin{align}
    \begin{split}\label{eq:app1}
    -2i\partial_t T^{p+1}_{(\bz,s)}[\psi](\bx,t) & \ = \sum_{\substack{\abs\bjx+j_t \leq p - 1}} \frac{-2i}{\bjx!j_t!} D^{(\bjx,j_t+1)} \psi(\bz,s) (\bx-\bz)^\bjx (t-s)^{j_t}\\ 
    &\stackrel{\eqref{eq:tsch}}{=}\sum_{\substack{\abs\bjx+j_t \leq p - 1}}\sum_{\ell = 1}^d \frac{1}{\bjx!j_t!} D^{(\bjx+2\be_\ell,j_t)} \psi(\bz,s) (\bx-\bz)^\bjx (t-s)^{j_t}  \\ 
    &\ = \Deltax T^{p+1}_{(\bz,s)}[\psi](\bx,t) + \sum_{\abs\bjx+j_t=p-1}\sum_{\ell = 1}^d \frac{1}{\bjx!j_t!} D^{(\bjx+2\be_\ell,j_t)} \psi(\bz,s) (\bx-\bz)^{\bjx} (t-s)^{j_t}.
    \end{split}
\end{align}
For the extended Taylor polynomial we have that
\begin{align*}
    \Deltax\ETaylor{(\bz,s)}{2p}{\psi}(\bx,t)
    &\ =\Deltax T^{p+1}_{(\bz,s)}[\psi](\bx,t) 
    + \sum_{\substack{2j_t+\abs\bjx\leq 2p-2\\ p-1 \leq j_t+\abs\bjx}} \sum_{\ell = 1}^d \frac{1}{\bjx!j_t!} D^{(\bjx+2 \be_\ell,j_t)} \psi(\bz,s) (\bx-\bz)^\bjx (t-s)^{j_t} \\
    &\ =\Deltax T^{p+1}_{(\bz,s)}[\psi](\bx,t) + 
    \sum_{\substack{j_t+\abs\bjx=p-1}} \sum_{\ell = 1}^d\frac{1}{\bjx!j_t!} D^{(\bjx+2 \be_{\ell},j_t)} \psi(\bz,s) (\bx-\bz)^\bjx (t-s)^{j_t} \\
    &\qquad\qquad\qquad\quad + \sum_{\substack{2j_t+\abs\bjx\leq 2p-2\\ p \leq j_t+\abs\bjx}} \sum_{\ell = 1}^d\frac{1}{\bjx!j_t!} D^{(\bjx+2 \be_{\ell} ,j_t)} \psi(\bz,s) (\bx-\bz)^\bjx (t-s)^{j_t} \\
    &\stackrel{\substack{\eqref{eq:tsch}\\ \eqref{eq:app1}}}{=} -2i\partial_t T^{p+1}_{(\bz,s)}[\psi](\bx,t) - \sum_{\substack{2j_t+\abs\bjx\leq 2p-2\\ p \leq  j_t+\abs\bjx}} \frac{2i }{\bjx!j_t!} D^{(\bjx,j_t+1)} \psi(\bz,s) (\bx-\bz)^\bjx (t-s)^{j_t} \\
    & \ =-2i\partial_t \ETaylor{(\bz,s)}{2p}{\psi}(\bx,t),
\end{align*}
and therefore $\ETaylor{(\bz, s)}{2p}{\psi}\in \bT(\Upsilon)$.
Since~$\ETaylor{(\bz, s)}{2p}{\psi}$ also belongs to~$\IP^{2p}(\Upsilon)$, then~$\ETaylor{(\bz, s)}{2p}{\psi} \in \IP^{2p}_{\bT}(\Upsilon)$. As the above identity is independent of the center point~$(\bz, s)$, the fact that~$\EAvTaylor{2p}{\psi} \in \IP^{2p}_{\bT}(\Upsilon)$ can be proven in a similar way.
\end{proof}

Combining Theorem~\ref{THEOREM::WELL-POSEDNESS}, Propositions~\ref{PROP::DGp-BOUND} and~\ref{PROP::EXTENDED-TAYLOR-POLYNOMIALS}, together with error bounds for extended averaged Taylor polynomials, which can be proven using that~$\EAvTaylor{2p}{\psi} = \AvTaylor{p+1}{\psi} + \mathcal{O}(\hK^{p+1})$, the following error estimate is obtained.

\begin{theorem}
\label{THM::ERROR-ESTIMATE-QUASI-TREFFTZ}
Given~$p\in\IN$, let~$\psi$ be the exact solution to~\eqref{EQN::SCHRODINGER-EQUATION} and~$\uhp\in \bThp(\Th)$ be the solution to the Trefftz-DG method~\eqref{EQN::VARIATIONAL-DG} with~$\bThp(\Th) := \prod_{K \in \Th} \IP^{2p}_{\bT}(K)$.
If~$\psi|_K\in \ESOBOLEV{p+1}{\Kt ; \ESOBOLEV{2p}{\Kx}}$ and~$\hKx \simeq \hKt$ for all~$K = \Kx \times \Kt \in \Th$, then there exists a positive constant~$C$ independent of the mesh size~$h$, but depending on the degree~$p$, 
the local quasi-uniformity parameter~$\mathsf{lqu}(\Th)$, and the measure of the space--time domain~$\QT$ such that
\begin{equation*}
\Tnorm{\psi - \uhp}{\DG} \leq C \sum_{K \in \Th} \max\{\hKx, \hKt\}^{p} \Norm{\psi}{\ESOBOLEV{p+1}{\Kt; \ESOBOLEV{2p}{\Kx}}}.
\end{equation*}
\end{theorem}

\begin{remark}
Theorem~\ref{THM::ERROR-ESTIMATE-QUASI-TREFFTZ} guarantees optimal~$h$-convergence rates, provided that, for~$p > 1$, the solution satisfy stronger regularity assumptions than those required for the exponential complex functions considered in~\cite{Gomez_Moiola_2022}. 
\end{remark}
\section{Basis and dimension of the local polynomial Trefftz space}
\noindent We study the dimension of the space~$\IP^{2p}_{\bT}(K)$ and provide a practical way to construct a basis, which is valid for any space dimension. Given~$p\in \IN$ and~$K = \Kx \times \Kt \in \Th$, let~$\{m_J\}_{J = 1}^{r_{d, 2p}}$ be a basis for the space~$\IP^{2p}{(\Kx)}$, where~$r_{d, 2p} = \dim(\IP^{2p}(\IR^d)) = \binom{2p+d}{d}$. For a fixed~$t_K \in \Kt$, we consider the following set of functions in~$\IP^{2p}_{\bT}(K)$
\begin{equation} 
\label{EQN::TREFFTZ-BASIS}
\mathcal{B}_{2p}^{\bT} := \left\{
b_J^\bT \in \IP^{2p}_{\bT}(K) : b_J(\bx, t_K) = m_J(\bx), \ J = 1, \ldots, r_{d, 2p}
\right\}.
\end{equation}
\begin{proposition}
    For any~$p \in \IN$ and~$K\in \Th$, the set~$\mathcal{B}_{2p}^{\bT}$ constitutes a basis for the space~$\IP^{2p}_{\bT}(K)$.
\end{proposition}
\begin{proof}
Let~$q_{2p}^{\bT} \in \IP^{2p}_{\bT}(K)$. Then, $q_{2p}^{\bT}$ can be expressed in the scaled monomial basis as
\begin{equation*}
    q_{2p}^{\bT}(\bx, t) = \sum_{\abs{\bj} \leq 2p} C_{\bj} \left(\frac{\bx - \bx_K}{\hKx}\right)^{\bjx} \left(\frac{t - t_K}{\hKt}\right)^{j_t}.
\end{equation*}
Since~$q_{2p}^{\bT} \in \bT(K)$, the complex coefficients~$\{C_{\bj}\}_{\abs{\bj}\leq 2p}$ must satisfy the following recurrence relation 
\begin{equation}\label{eq:recurring}
  C_{\bjx, j_t + 1} = 
  \begin{cases}
\frac{i \hKt}{2(j_t + 1)\hKx^2} \sum_{\ell = 1}^d  (\bjxl + 1)(\bjxl + 2) C_{\bjx + 2\be_\ell, j_t}  & \text{ if } \abs{\bjx} \leq 2p - 2, \\
0 & \text{ otherwise}.
\end{cases}
\end{equation}
As a consequence, once the coefficients~$C_{\bjx, 0}$ are fixed, the remaining coefficients are determined by~\eqref{eq:recurring}.
This implies that each element in~$\IP^{2p}_{\bT}(K)$ is uniquely identified by its restriction to~$t = t_K$. 

On the other hand, by the definition of~$\mathcal{B}_{2p}^\bT$, there exists some coefficients~$\{\gamma_J\}_{J = 1}^{r_{d, 2p}} \subset \IC$ such that
\begin{equation*}
q_{2p}^{\bT}(\bx, t_K) = \sum_{J = 1}^{r_{d, 2p}} \gamma_J m_J(\bx) = \sum_{J = 1}^{r_{d, 2p}} \gamma_J b_J(\bx, t_K) \quad \Longrightarrow \quad 
q_{2p}^{\bT} = \sum_{J = 1}^{r_{d, 2p}} \gamma_J b_J.
\end{equation*}
Finally, the linear independence of the set~$\mathcal{B}_{2p}^{\bT}$ follows from the restriction of its elements to~$t = t_K$, and the linear independence of the basis~$\{m_J\}_{J = 1}^{r_{d, 2p}}$.
\end{proof}
 
To construct a basis for the local polynomial Trefftz space~$\mathcal{B}_{2p}^{\bT}$ one can now employ the recurrence relation~\eqref{eq:recurring}.
To initialize the recursion, one chooses a basis~$\{m_J\}_{J = 1}^{r_{d, 2p}}$ of~$\IP^{2p}(\Kx)$.
Natural choices are (scaled and/or translated) monomials, Legendre, or Chebyshev
polynomials. For instance, if~$d = 1$ and we set~$\{m_J\}_{J = 1}^{r_{1, 2p}}$ as monomials, then~$\IP^2_{\bT} = \spn\{1, x, it + x^2\}$ and~$\IP^4_{\bT} = \spn\{1, x, it + x^2, x^3+3ixt, x^4+6ix^2t-3t^2\}$. 

\section{Numerical results}
\noindent We validate the error estimate of Theorem~\ref{THM::ERROR-ESTIMATE-QUASI-TREFFTZ}, and assess numerically some additional features of the method. 

\subsection{Smooth solution}
\noindent On the space--time domain $\QT = (0, 1) \times (0, 1)$ we consider a manufactured~$(1+1)$-dimensional Schr\"odinger equation with initial and Dirichlet boundary conditions so that the exact solution is given by the complex wave function~$\psi(x, t) = \exp\left(\kappa x + \frac{i\kappa^2t}{2}\right),$
for some given~$\kappa \in \IR$.
For the construction of the Trefftz basis functions~$\{b_J^{\bT}\}_{J=1}^{r_{1, 2p}}$, we consider two choices of~$\{m_J\}_{J = 1}^{r_{1, 2p}}$ in~\eqref{EQN::TREFFTZ-BASIS}:
\begin{subequations}
\begin{align}
\label{EQN::CHOICE-BASIS-1}
m_J(x) & := \left(\frac{x - x_K}{h_x}\right)^J \qquad J = 1, \ldots, 2p + 1, \\
\label{EQN::CHOICE-BASIS-2}
m_J(x) & := \frac{(x - x_K)^J}{h_x^{\lfloor J/2 \rfloor}} \quad \qquad J = 1, \ldots, 2p + 1,
\end{align}
\end{subequations}
where~$\lfloor \cdot \rfloor$ is the floor function. 
The remaining coefficients~$C_{\bj}$ are computed with the relations~$\eqref{eq:recurring}$. No difference in terms of accuracy have been observed for these two choices.

In Figure~\ref{FIG::FREE-PARTICLE-H}, we show the errors of method~\eqref{EQN::VARIATIONAL-DG} in the DG norm obtained for~$k = 5$ and a sequence of meshes with $h_t = h_x = 0.1 \times 2^{-j}, \ j = 0, \ldots, 4$. 
Optimal rates of convergence of order~$\ORDER{h^{p}}$ are observed as predicted by Theorem~\ref{THM::ERROR-ESTIMATE-QUASI-TREFFTZ}. 
In Figure~\ref{FIG::FREE-PARTICLE-P}, we observe exponential convergence of order~$\ORDER{\text{e}^{-b N_{DoFs}}}$ for the~$p$-version of the method, i.e., by fixing the space--time mesh and increasing the degree of accuracy~$p$.
Such an exponential convergence was also observed for the non-polynomial Trefftz space considered in~\cite{Gomez_Moiola_2022}. 
In Figures~\ref{FIG::FREE-PARTICLE-CONDITIONING-1} and~\ref{FIG::FREE-PARTICLE-CONDITIONING-2}, we show the condition number of the matrices stemming from method~\eqref{EQN::VARIATIONAL-DG} for different values of~$p$ and both choices of~$\{m_J\}_{J = 1}^{r_{1, 2p}}$, which grow as~$\ORDER{h^{-(2p +1)}}$ and~$\ORDER{h^{-1}}$ for the choices in~\eqref{EQN::CHOICE-BASIS-1} and~\eqref{EQN::CHOICE-BASIS-2}, respectively.
The former behaviour was observed when an arbitrary choice of the parameters defining the basis functions of the non-polynomial Trefftz space was used in~\cite{Gomez_Moiola_2022}, while the latter behaviour was observed for a choice that produces an orthogonal basis, see~\cite{Gomez_Moiola_2023}. 
This highlights the relevance of using suitable bases for the discrete Trefftz spaces.
\begin{figure}[!ht]
    \centering
    \subfloat[$h$-convergence \label{FIG::FREE-PARTICLE-H}]{
        \includegraphics[width = .45\textwidth]{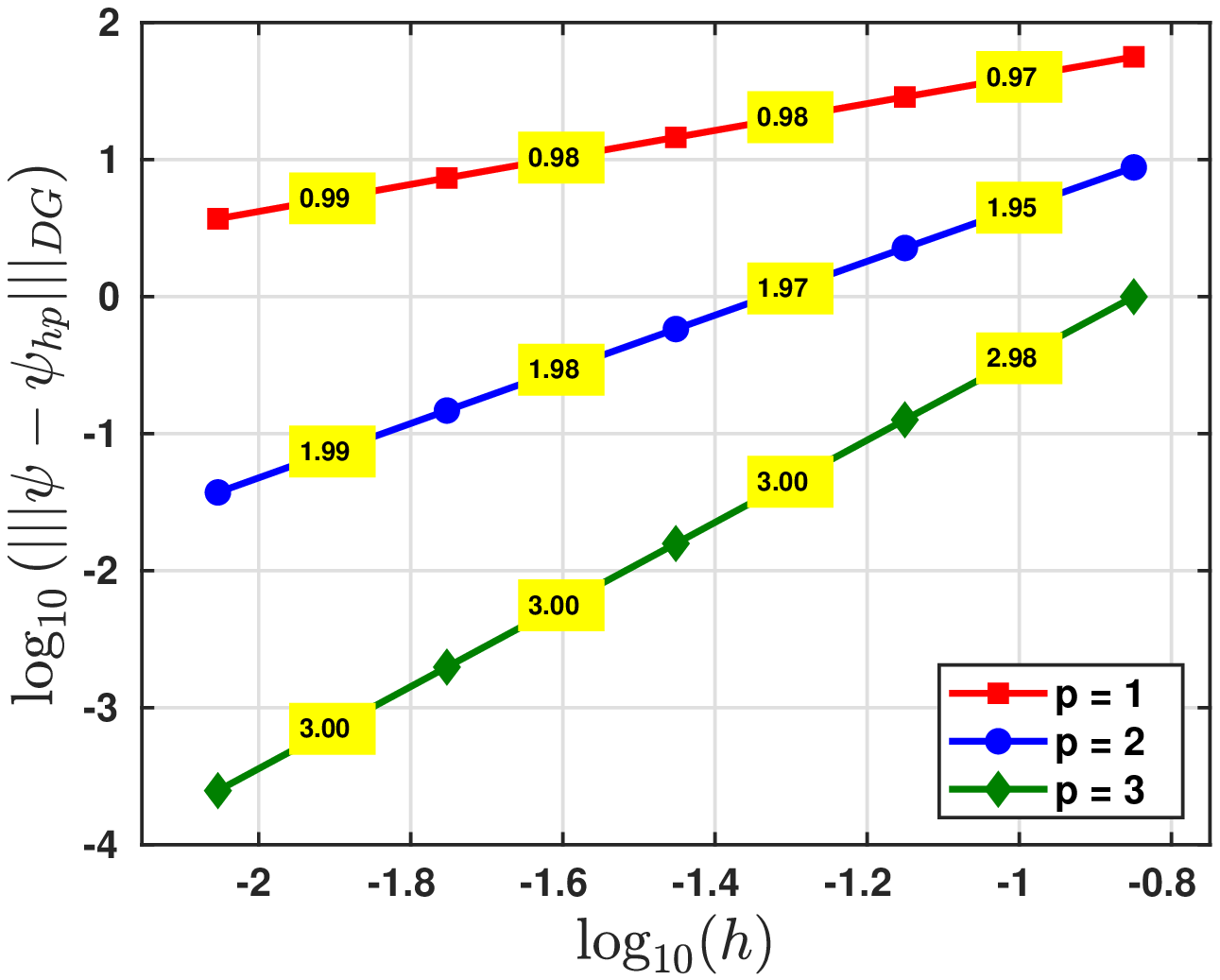}
    }
    \hspace{0.1in}
    \subfloat[$p$-convergence \label{FIG::FREE-PARTICLE-P}
    ]{
        \includegraphics[width = .45\textwidth]{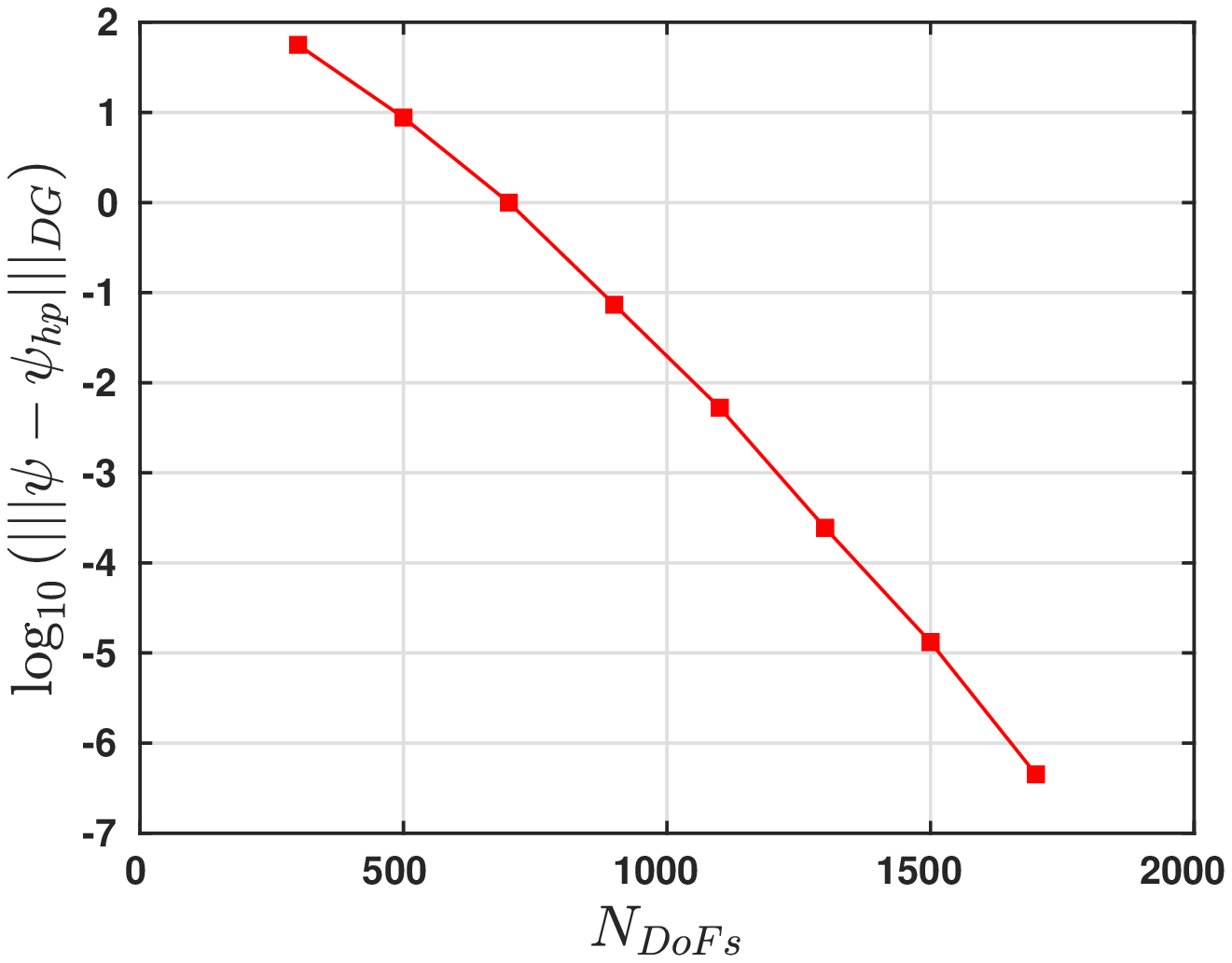}
    }\\
    \subfloat[Conditioning for~$\{m_J\}_{J = 1}^{r_{1, 2p}}$ in~\eqref{EQN::CHOICE-BASIS-1} \label{FIG::FREE-PARTICLE-CONDITIONING-1}]{
        \includegraphics[width = .45\textwidth]{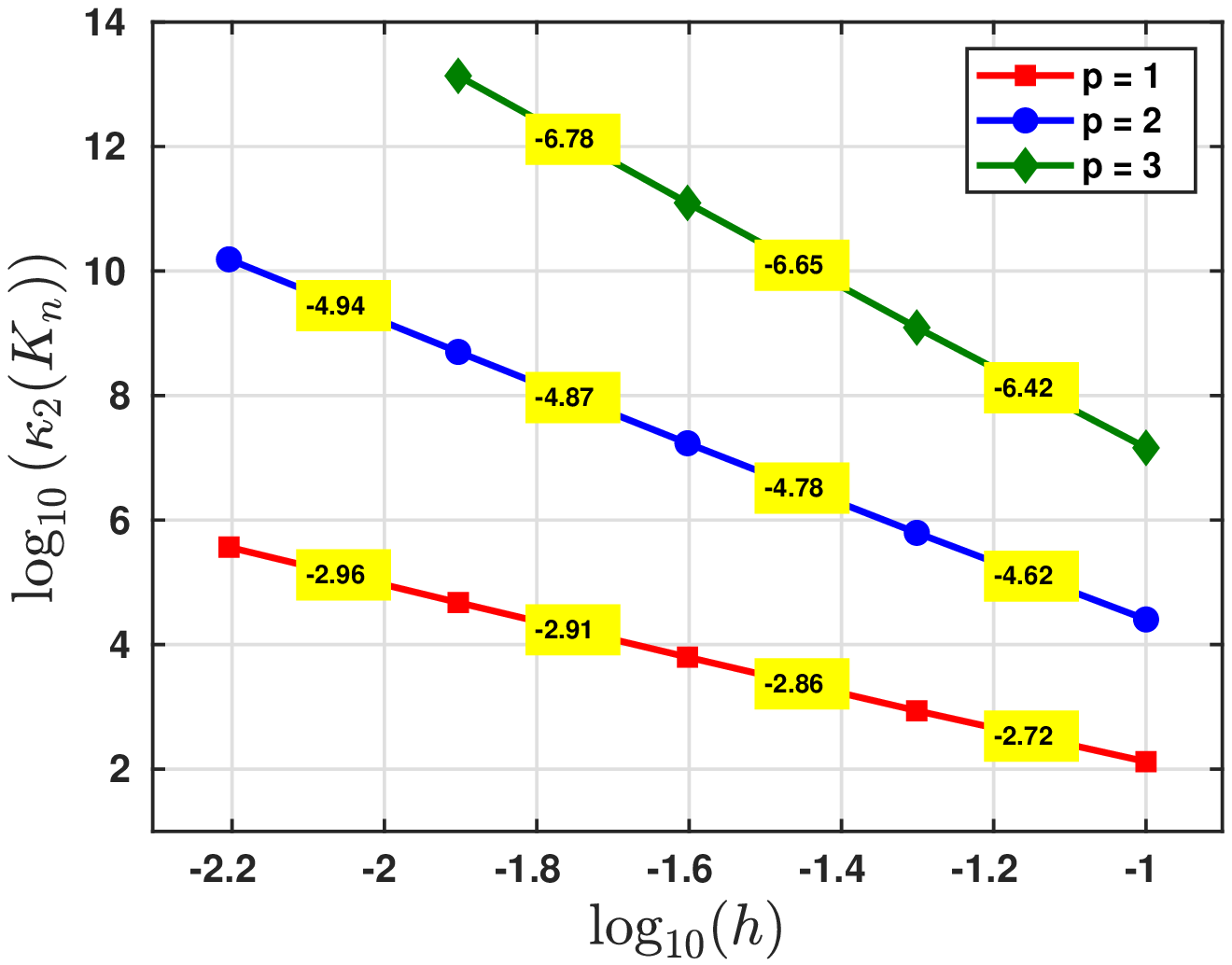}
    }
    \hspace{0.1in}
    \subfloat[Conditioning for~$\{m_J\}_{J = 1}^{r_{1, 2p}}$ in~\eqref{EQN::CHOICE-BASIS-2} \label{FIG::FREE-PARTICLE-CONDITIONING-2}]{
        \includegraphics[width = .45\textwidth]{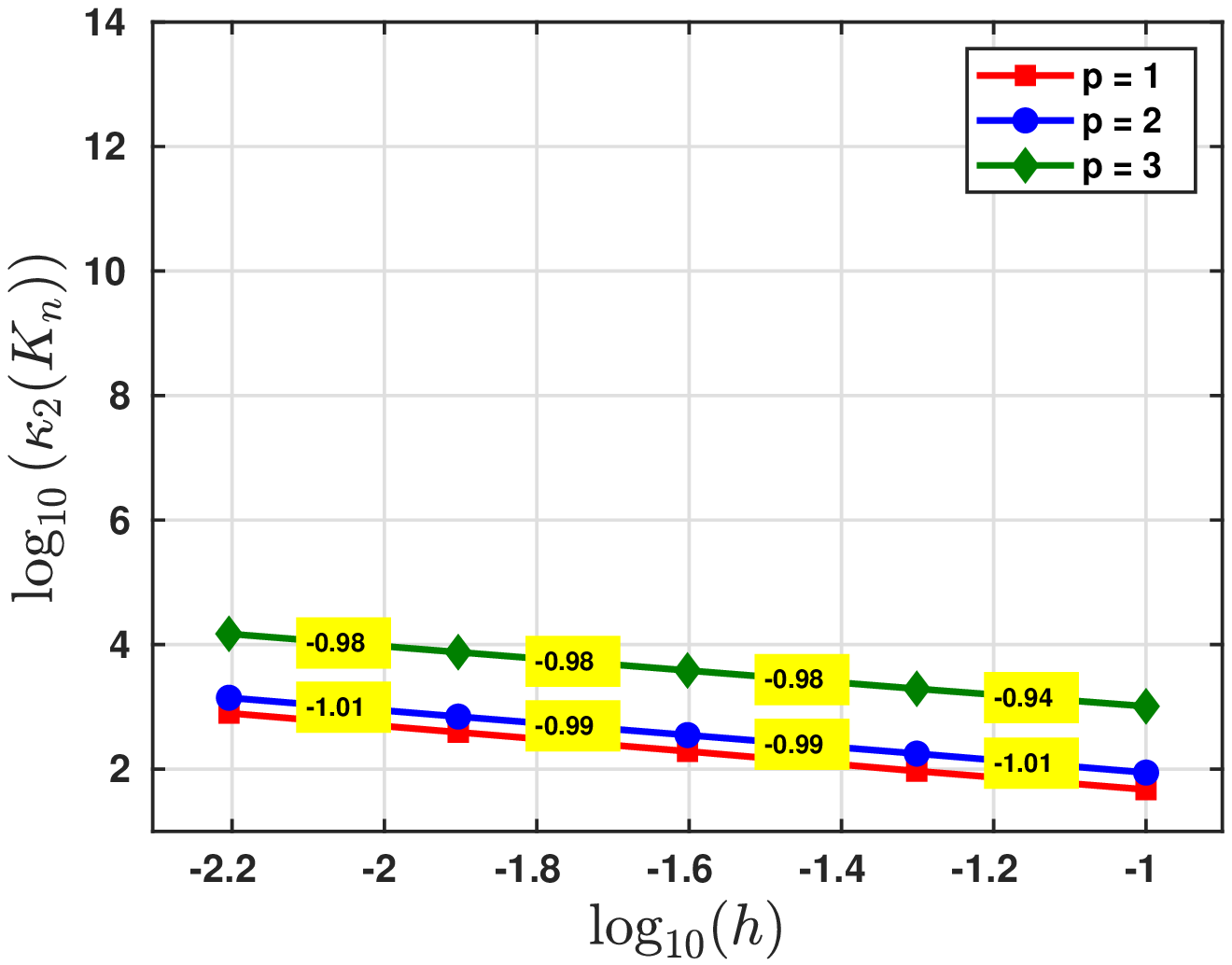}
    }
    \caption{Numerical results for the~$(1+1)$-dimensional problem with exact solution~$\psi(x, t) = \exp\left(\kappa x + \frac{i\kappa^2t}{2}\right)$, for~$\kappa = 5$. \label{FIG::SQUARE-WELL}} 
\end{figure}

\subsection{Singular solution}
\noindent We consider the~$(1+1)$-dimensional problem on the space--time domain~$\QT = (0, 1) \times (0, 0.1)$ with homogeneous Dirichlet boundary conditions, and 
initial condition~$\psi_0(x) = Ax(1 - x)$ with normalization constant~$A = \sqrt{30}$. The exact solution is given by (see~\cite[Example~2.2, Ch.~2]{Griffiths_2018})
\begin{equation}  
\label{EQN::SINGULAR-SOL}
\displaystyle \psi(x, t) = \sqrt{30} \left(\frac{2}{\pi}\right)^3 \sum_{n = 1}^{\infty} \frac{1}{(2n+1)^3} \sin\big((2n + 1)\pi x\big) e^{\frac{-i(2n +1)^2 \pi^2}{2}},
\end{equation}
which belongs to~$H^{\frac54-\epsilon}(0, T; H_0^1(\Omega))$ for all~$\epsilon > 0$; cf.~\cite[\S~7.1]{Schotzau_Schwab_2000}. 
In Figure~\ref{FIG::SINGULAR-SOL}, we show the errors in the~DG norm obtained for a sequence of meshes with~$h_t = 0.1 h_x = 0.05 \times 2^{-j}$, $j = 0, \ldots, 4$, and different discrete spaces: the polynomial Trefftz space~$\IP_{\bT}^{2p}$ (Fig.~\ref{FIG::SINGULAR-POLY-TREFFTZ}); the quasi-Trefftz polynomial space introduced in~\cite{Gomez_Moiola_2023} (Fig.~\ref{FIG::SINGULAR-QUASI-TREFFTZ}); the full polynomial space~$\IP^p$ (Fig.~\ref{FIG::SINGULAR-FULL-POLY}); the pseudo-plane wave Trefftz space introduced in~\cite{Gomez_Moiola_2022} (Fig.~\ref{FIG::SINGULAR-PW}). We present some details for this experiment:
\begin{itemize}
\item The polynomial quasi-Trefftz space~$\mathbb{Q}\mathbb{T}^p(K)$ was defined in~\cite{Gomez_Moiola_2023} as
\begin{equation*}
\mathbb{Q}\mathbb{T}^p(K) := \left\{q_p \in \IP^p(K) : D^{\bj} \calS q_p (x_K, t_K) = 0, \ \abs{\bj} \leq p - 2\right\},
\end{equation*}
for some~$(x_K, t_K) \in K$, which in this experiment is set as the center of~$K$. In~$(1+1)$ dimensions, the quasi-Trefftz space~$\QTrefftz{p}{K}$ has the same dimension ($2p+1$) as~$\IP_{\bT}^{2p}(K)$, but it does not reduce to a polynomial Trefftz space when the potential is zero.
\item In~$(1+1)$ dimensions, the pseudo-plane wave Trefftz space~$\bTp^p(K)$ in~\cite{Gomez_Moiola_2022} is given by
\begin{equation*}
\bTp^p(K) := \text{span}\left\{\phi_{\ell} = \exp\Big(i\big(k_{\ell}x - \frac12 k_{\ell}^2t\big)\Big), \ \ell = 1, \ldots, 2p+1\right\},
\end{equation*}
for some real parameters~$k_{\ell}$. In this numerical experiment, we have set~$k_{\ell} = -2p + 2(\ell - 1),$ for $\ell = 1, \ldots, 2p+1$. For this space, the error for the finest mesh and~$p = 2$ could not be computed due to the ill-conditioning of its stiffness matrix.
\item For the quasi-Trefftz and full-polynomial spaces, we have used the ultra-weak space--time DG variational formulation in~\cite{Gomez_Moiola_2023}, which extends the one in~\eqref{EQN::VARIATIONAL-DG} to non-Trefftz discrete spaces.
\item To compute the errors, we have truncated the series in~\eqref{EQN::SINGULAR-SOL} to~$n = 250$.
\end{itemize}
Slightly reduced rates of convergence are observed for the polynomial and the pseudo-plane wave Trefftz spaces, compared to the quasi-Trefftz and full-polynomial spaces; however, for the above parameter choices, the discrete Trefftz spaces give the smallest errors. 
Such a reduction of the convergence rates is expected for the polynomial Trefftz space, as the error estimate in Theorem~\ref{THM::ERROR-ESTIMATE-QUASI-TREFFTZ} requires stronger regularity assumptions in space on the exact solution. This motivates to look for sharper error estimates, especially for singular problems.

\begin{figure}[!ht]
    \centering
    \subfloat[Polynomial Trefftz space \label{FIG::SINGULAR-POLY-TREFFTZ}]{
        \includegraphics[width = .45\textwidth]{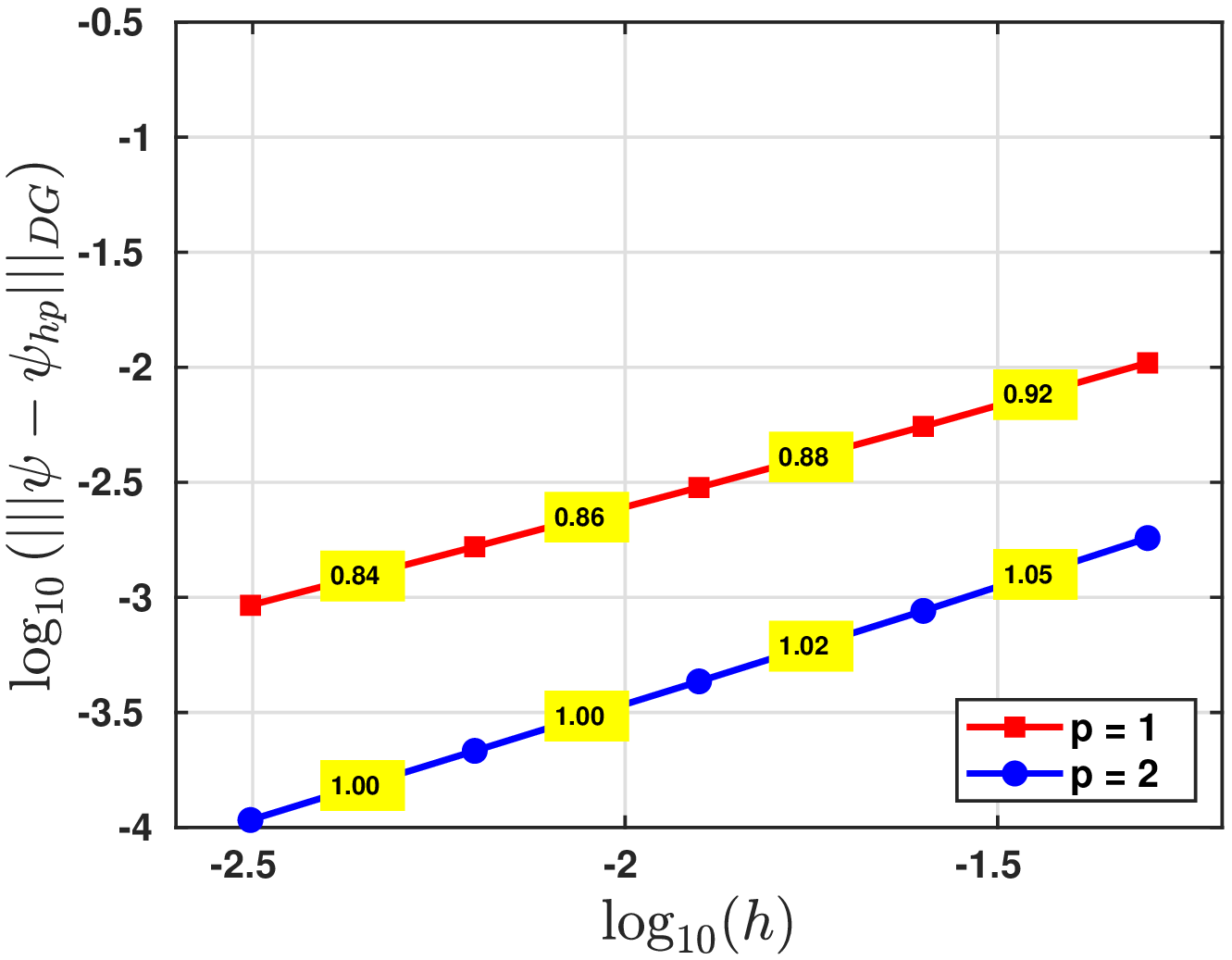} 
    } 
    \hspace{0.1in}
    \subfloat[Quasi-Trefftz space \label{FIG::SINGULAR-QUASI-TREFFTZ}
    ]{
        \includegraphics[width = .45\textwidth]{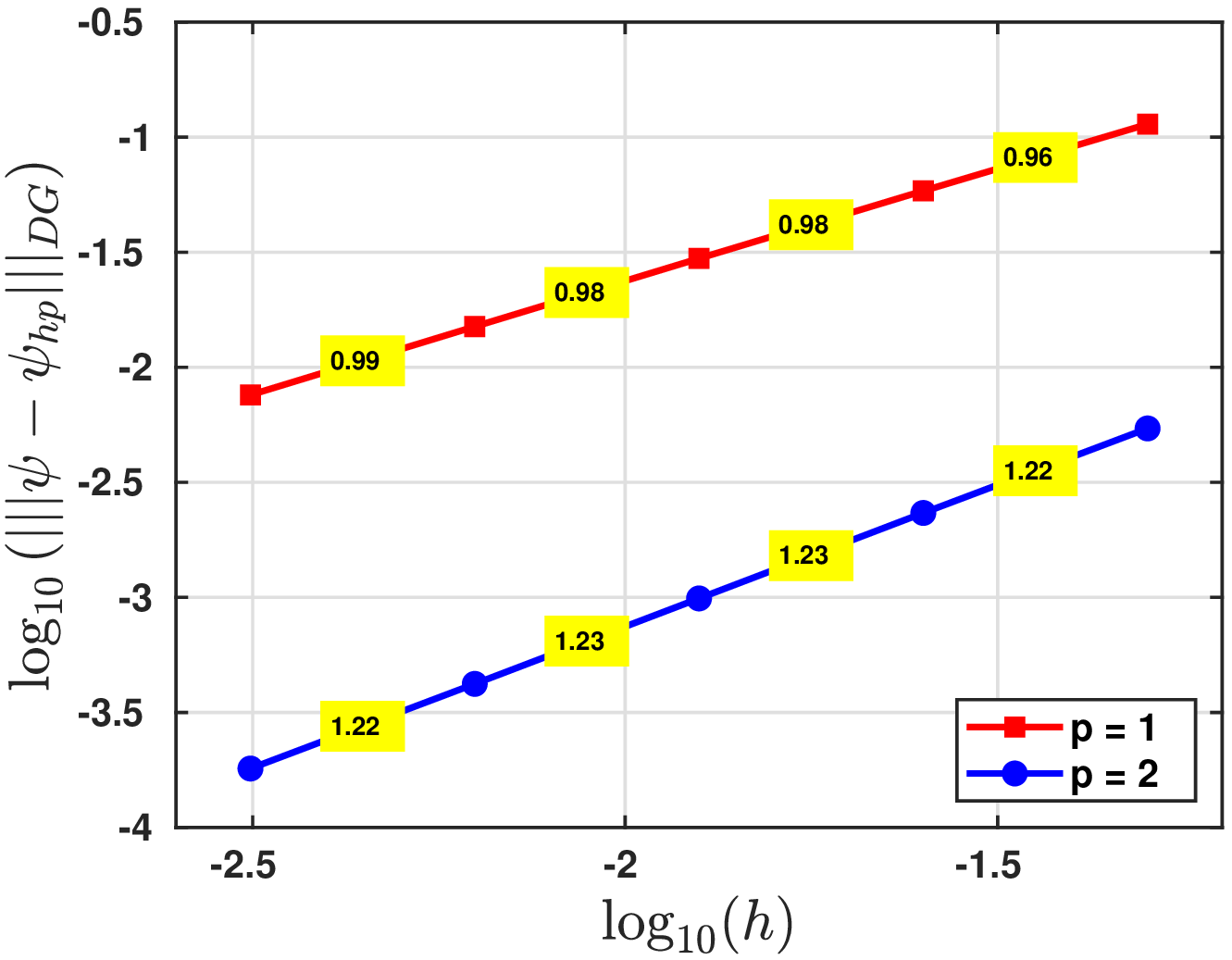}
    }\\
    \subfloat[Full polynomial space \label{FIG::SINGULAR-FULL-POLY}]{
        \includegraphics[width = .45\textwidth]{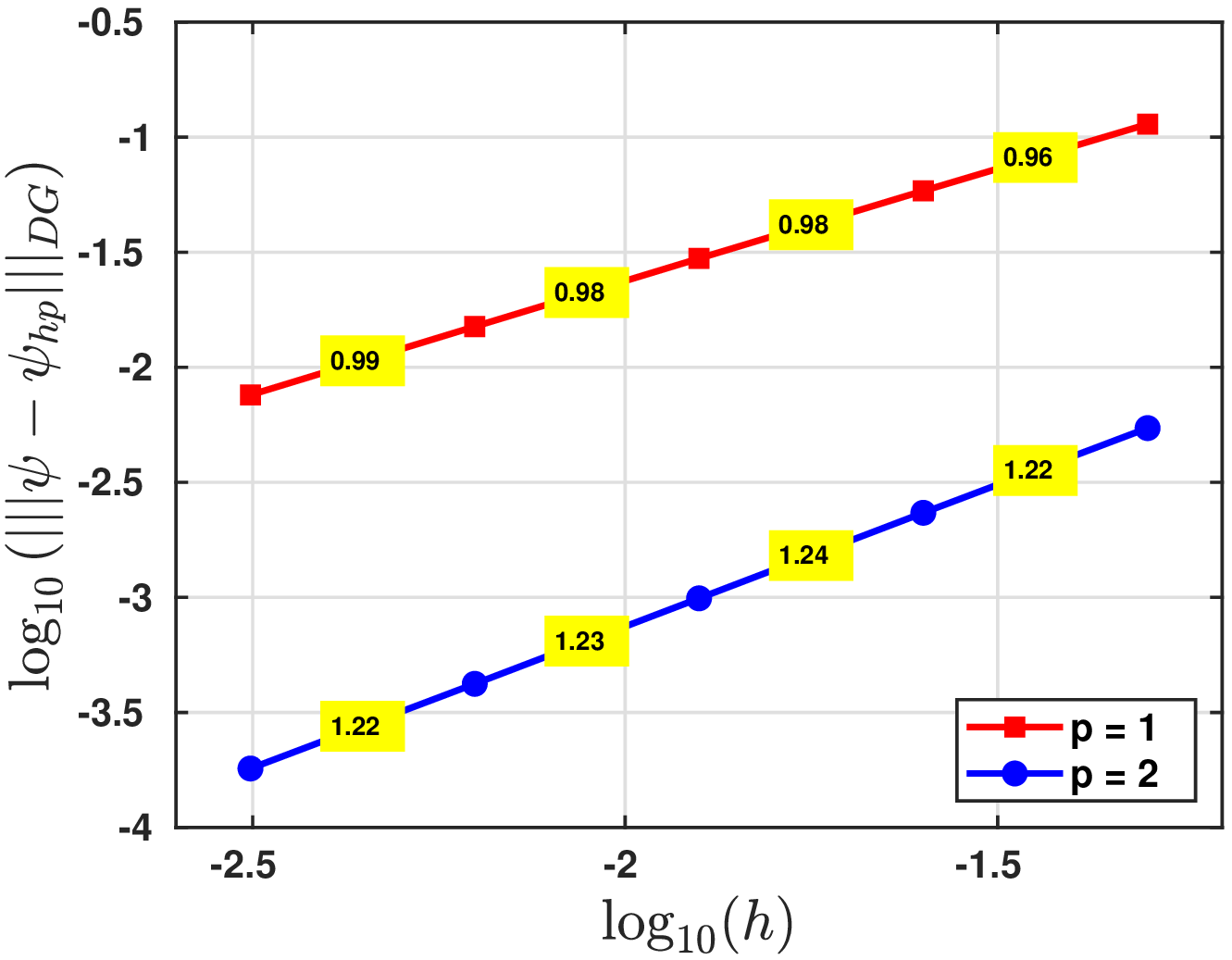}
    } 
    \hspace{0.1in}
    \subfloat[Pseudo-plane wave Trefftz space \label{FIG::SINGULAR-PW}]{
        \includegraphics[width = .45\textwidth]{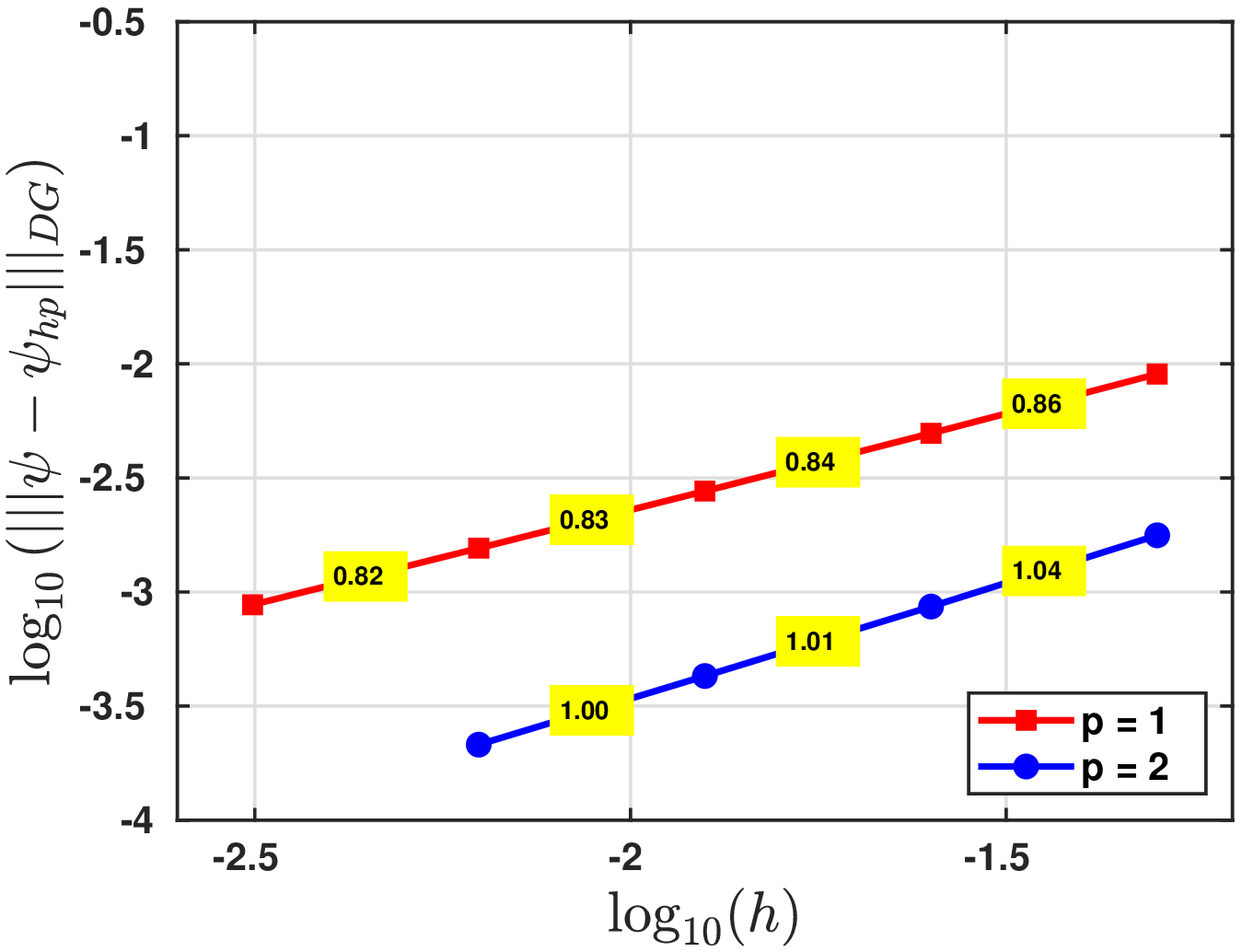}
    }
    \caption{$h$-convergence for the~$(1+1)$-dimensional problem with singular solution~$\psi$ in~\eqref{EQN::SINGULAR-SOL} and different discrete spaces. \label{FIG::SINGULAR-SOL}} 
\end{figure}
\section{Conclusions}
 \noindent We have studied the approximation properties of polynomial Trefftz spaces for the time-dependent linear Schr\"odinger equation in any space dimension~$d\in \IN$.
We have proven that, if Trefftz polynomials of degree~$2p$ are used in the space--time Trefftz-DG formulation of~\cite{Gomez_Moiola_2022},  optimal~$h$-convergence of order~$\ORDER{h^p}$ is obtained for the error in a mesh-dependent norm. For~$d = 1$, the dimension of the polynomial Trefftz space of degree~$2p$ is always smaller than that of the full polynomial space of degree~$p$. However, for~$d > 1$, the reduction in the total number of degrees of freedom
takes place only for large values of~$p$, e.g., $p > 7$ for~$d = 2$ and~$p > 24$ for~$d = 3$. 

\section*{Acknowledgements}
\noindent We are very grateful to Matteo Ferrari for his suggestion to use~\eqref{EQN::CHOICE-BASIS-2} in order to reduce the condition number of the stiffness matrices.


\end{document}